\documentclass[11pt,twoside,a4paper]{article}
\usepackage{amsmath,amssymb,mathrsfs}
\usepackage{amsthm}

\usepackage{inputenc}
\usepackage{avant}
\usepackage[english]{babel}
\usepackage{
thmtools,thm-restate}
\usepackage[nottoc
]
{tocbibind}
\usepackage{graphicx}
\usepackage{enumitem}
\usepackage{tikz}

\usepackage{bbm}
\usepackage{esint}
\usepackage{authblk}
\usepackage{mathtools}
\usepackage{dsfont}
\usepackage{interval}
\intervalconfig{
soft open fences}

\usepackage{nicematrix}

\NiceMatrixOptions{cell-space-limits = 1pt}

\usepackage[inner=2.4cm,outer=2.4cm,top=2.8cm,bottom=2.8cm]{geometry}
\usepackage[all]{xy}
 \usepackage{todonotes}
 \newcommand
{\red}{
\color{red}}
 \newcommand{\blue}{\color{blue}}

\usepackage[bookmarks=true]{hyperref}
\usepackage{xcolor}
\hypersetup{
    colorlinks,
    linkcolor={red!50!black},
    citecolor={blue!50!black},
    urlcolor={blue!80!black},
}

\mathtoolsset{showonlyrefs}  

\DeclareMathOperator{\diam}{diam}
\DeclareMathOperator{\supp}{supp}

\newcommand{\norm}[1]{\left\lVert#1\right\rVert}
\newcommand{\normdot}{{\|\!\cdot\!\|}}

\newcommand{\normc}[1]{\left\lvert#1\right\rvert}
\newcommand{\normdotc}{{|\!\cdot\!|}}

\newcommand{\Leb}{\mathscr{L}}

\newcommand{\N}{\mathbb{N}}
\newcommand{\R}{\mathbb{R}}

\newcommand{\de}{\ensuremath{\, \mathrm d}} 

\newcommand\restr[2]{{
  \left.\kern-\nulldelimiterspace 
  #1 
  \right|_{#2} 
  }}

\DeclarePairedDelimiter\abs{\lvert}{\rvert}%

\makeatletter
\let\oldabs\abs
\def\abs{\@ifstar{\oldabs}{\oldabs*}}
\newcommand{\cd}{\mathsf{CD}}

\newcommand{\MCP}{\mathsf{MCP}}

\newcommand{\X}{\mathsf{X}}

\newcommand{\di}
{\mathsf d} 
\newcommand{\m}{\mathfrak m} 

\newcommand{\dis}{\mathcal D}

\newcommand{\sF}{sub-Finsler }
\newcommand{\sr}{sub-Riemannian }

\newcommand{\hei}{\mathbb{H}}
\newcommand{\Haus}{\mathscr{H}}

\renewcommand{\epsilon}{\varepsilon}

\usepackage{titlesec}
\titleformat{\section}
  {\scshape\large\centering}
  {\thesection}{0.5em}{}
\titleformat{\subsection}
  {\scshape\centering}
  {\thesubsection}{0.4em}{}

\title{\textbf{A review of the tangent space in sub-Finsler geometry and applications to the failure of the $\cd$ condition}}
\date{\today}

\author{ Mattia Magnabosco\footnote{Mathematical Institute, University of Oxford. \textit{E-mail}:  \href{mailto:mattia.magnabosco@maths.ox.ac.uk}{mattia.magnabosco@maths.ox.ac.uk}}\, \text{and}\, 
Tommaso Rossi\footnote{Laboratoire Jacques-Louis Lions, Sorbonne Universit\'e. \textit{E-mail}: \href{mailto:tommaso.rossi@inria.fr}{tommaso.rossi@inria.fr}} }

\newtheoremstyle{remark}
        {10pt}
        {10pt}
        {}
        {}
        {\itshape}
        {.}
        {.4em}
        {}

\newtheoremstyle{proof}
        {10pt}
        {10pt}
        {}
        {}
        {\itshape}
        {.}
        {.4em}
        {}
        
\newtheoremstyle{definition}
        {10pt}
        {10pt}
        {}
        {}
        {\bfseries}
        {.}
        {.4em}
        {}

\newtheoremstyle{theorem}
        {10pt}
        {10pt}
        {\slshape}
        {}
        {\bfseries}
        {.}
        {.4em}
        {}

\theoremstyle{theorem}
\newtheorem{theorem}{Theorem}[section]
\newtheorem{prop}[theorem]{Proposition}
\newtheorem{corollary}[theorem]
{Corollary}
\newtheorem{lemma}[theorem]{Lemma}

\theoremstyle{definition}
\newtheorem{definition}[theorem]{Definition}

\theoremstyle{remark}
\newtheorem{remark}[theorem]{Remark}

\theoremstyle
{proof}
\newtheorem*{pro}{Proof}

 {\popQED\end{pro}}

\makeatletter
\renewcommand\xleftrightarrow[2][]{%
  \ext@arrow 9999{\longleftrightarrowfill@}{#1}{#2}}
\newcommand\longleftrightarrowfill@{%
  \arrowfill@\leftarrow\relbar\rightarrow}
\makeatother

\makeindex

\begin{document}

\maketitle

\begin{abstract}
We review the construction of the tangent space to a \sF manifold in the measured Gromov-Hausdorff sense. Under suitable assumptions on the measure, the metric measure tangent is described by the nilpotent approximation, equipped with a scalar multiple of the Lebesgue measure. We apply this result in the study of the Lott--Sturm--Villani curvature-dimension condition in sub-Finsler geometry. In particular, we show the failure of the $\cd$ condition in 3D-contact \sF manifolds, equipped with a bounded measure. 
\end{abstract}



\section{Introduction}

The celebrated curvature-dimension condition (or $\cd$ for short) à-la Lott--Sturm--Villani \cite{sturm2006-1,sturm2006,lott--villani2009} is a synthetic notion of Ricci curvature lower bounds (and dimensional upper bounds). More precisely, for a metric measure space, the $\cd(K,N)$ condition asks for a suitable $(K,N)$-convexity property of the R\'enyi entropy functionals in the Wasserstein space (see \cite[Def.\ 1.3]{sturm2006} for the precise definition). Here $K\in\R$ represents a lower bound on the (Ricci) curvature while $N\in(1,\infty]$ can be thought of as an upper bound on the dimension. Additionally, in a smooth Riemannian manifold, the $\cd$ condition is a consistent curvature-dimension bound, in the sense it is equivalent to a lower bound on the Bakry-\'Emery Ricci curvature tensor. 

A similar relation between curvature and $\cd$ condition does not hold in the \sr setting and \sF setting. Sub-Finsler geometry is a generalization of Finsler and sub-Riemannian geome\-try, where, given a smooth manifold $M$, a smoothly varying norm is defined only on a set of preferred directions, called the distribution. 
Under the so-called H\"ormander condition, the length-minimization procedure among admissible curves (namely those curves that are tangent to the distribution) yields a well-defined distance on $M$, called \sF distance. 
In the contributions \cite{Juillet2020,MR4562156,MR4623954}, it shown that every complete \sr manifold, equipped with a positive smooth measure, cannot satisfy the $\cd(K,N)$ condition for any choice of the parameters $K$ and $N$. The same result has been proved for the \sF Heisenberg group \cite{borzatashiro,borza2024measure,borza2024curvature} and for smooth \sF manifolds \cite{magnabosco2023failure}. In this paper, we generalize these results in two directions. First of all, we relax the assumptions on the measure by allowing bounded reference measures, in the sense of Definition \ref{def:bounded_measure}. Second of all, we prove the failure of the $\cd$ condition for 3D-contact \sF manifolds. We refer to Section \ref{sec:contact} for the precise statements of the results.

The strategy of the proofs is based on the stability of the $\cd$ condition with respect to the pointed measured Gromov-Hausdorff convergence. In particular, given a \sF manifold $M$, equipped with a bounded measure $\m$, we assume by contradiction that the $\cd(K,N)$ condition holds. We then pass to the tangent space in the pointed measured Gromov-Hausdorff sense. By stability of the $\cd$ condition, the metric measure tangents of $(M,\di_{SF},\m)$ satisfy the $\cd(0,N)$ condition. Therefore, to obtain a contradiction, it is enough to show that there exists a metric measure tangent of $(M,\di_{SF},\m)$ that does not satisfy the $\cd$ condition. 

The last step of the strategy is based on the characterization of the metric measure tangents to a \sF manifold. 
Loosely speaking, for a general metric measure space $(\X,\di,\m)$, a metric measure tangent at $x\in \X$ captures its infinitesimal geometry around the base point $x$.
In a Riemannian manifold $(M,\mathrm g)$, the metric measure tangent space at a point $x\in M$ coincides with $T_xM$, equipped with the scalar product $\mathrm g_x$ and with the Lebesgue measure. 

In a \sF manifold, identifying the tangent space becomes less trivial, as the infinitesimal model is instead a quotient of a \sF Carnot group\footnote{A Carnot group is a connected and simply connected Lie group with stratified Lie algebra.}. 
In \cite{MR806700,MR1421822}, it is shown that the metric tangent to a complete \sr manifold $M$ at a point $x\in M$ is identified by its nilpotent approximation at $x$. The nilpotent approximation can be regarded as a non-holonomic first-order approximation of the \sr structure at $x$ and, in turn, is isometric to a quotient of a \sr Carnot group (where a scalar product is defined on the first layer of its Lie algebra). In \cite{MR4645068}, the analogous result is proved in the greater generality of Lipschitz \sF manifolds with continuously varying norms (including also smooth \sF manifolds). For the convergence of the measure, adapting and extending the results of \cite{MR3388884,SRmeasures}, we characterize the blow-up of a general bounded measure on a \sF manifold. Here a bounded measure is a Borel measure which, in coordinates, is absolutely continuous with respect to the Lebesgue measure, with density bounded from above and below by strictly positive constants, see Definition \ref{def:bounded_measure}. The description of the metric measure tangents of a \sF manifold is the summarized in the following theorem.

\begin{theorem}\label{thm:tangente}
     Let $M$ be a complete \sF manifold, equipped with a bounded measure $\m$. Then, for $\m$-a.e. $q\in M$, there is a unique metric measure tangent space of $M$ at $q$ and it is its nilpotent approximation equipped with (a multiple of) the Lebesgue measure. 
 \end{theorem}

\paragraph{Structure of the paper.} In Section \ref{sec:nilpotent}, we describe the nilpotent approximation of a \sF manifold. In Section \ref{sec:tangent}, we identify the blow-up of the bounded measure $\m$, leveraging its Lebesgue points, in the sense of Definition \ref{def:lebesgue_points}. In detail, we show that $\m$-almost every point is Lebesgue and, at those points, the blow-up is a scalar multiple of the Lebesgue measure of the correct dimension. Along the way, we also show that on equiregular manifolds, a measure is bounded if and only if it is locally Ahlfors regular. Finally, we prove Theorem \ref{thm:tangente}. We conclude the paper by detailing the aforementioned applications to the failure of the $\cd$ condition, see Section \ref{sec:contact}.

\subsection*{Acknowledgments} 
M.M. acknowledges support from the Royal Society through the Newton International Fellowship (award number: NIF$\backslash$R1$\backslash$231659). T.R. acknowledges support from the ANR-DFG project ``CoRoMo'' (ANR-22-CE92-0077-01). We wish to thank the organizers of the conference ``Metric Measure Spaces, Ricci Curvature, and Optimal Transport'' for the beautiful event. 

\section{Preliminaries}

Let $M$ be a smooth manifold of dimension $n$ and let $k\in\N$. A \emph{\sF structure} on $M$ is a couple $(\xi,\normdotc)$, such that: 
\begin{enumerate}[label=(\roman*)]
    \item $\xi: E\rightarrow TM$ is a morphism of vector bundles, where $E$ is a normed vector bundle over $M$, meaning that, for every $q\in M$, the fiber $E_q$ is a Banach space with norm $\normdotc_q$. In addition, the dependence $q\mapsto \normdotc_q$ is smooth.
    \item  The set of horizontal vector fields, defined as 
    \begin{equation}
    \dis := \big\{ \xi\circ\sigma \, :\, \sigma \in \Gamma(E) \big\} \subset \Gamma(TM),
    \end{equation}
    is a \emph{bracket-generating} family of vector fields (or it satisfies the H\"ormander condtion), namely, setting
    \begin{equation}
        {\rm Lie}_q(\dis):=\big\{X(q)\,:\, X\in\text{span}\{[X_1,\ldots,[X_{j-1},X_j]]\,:\, X_i\in\dis,j\in\N\} \big\},\qquad\forall\,q\in M,
    \end{equation}
    it holds that ${\rm Lie}_q(\dis)=T_qM$, for every $q\in M$.
\end{enumerate}
We say that $(M,\xi,\normdotc)$, or simply $M$ if there is no ambiguity, is a \sF manifold. 

\begin{remark}
    Every \sF
 structure $(\xi,\normdotc)$ is equivalent to a free one, meaning that we can (and will) always assume that $E=M\times\R^k$, for some $k\in \N$. See \cite[Sec.\ 3.1.4]{ABB-srgeom} for details. However, differently from the \sr setting, we can not ensure that the norm $\normdotc_q$ on the fiber $E_q=\{q\}\times\R^k$ is independent of $q$.    
\end{remark}

At every point $q\in M$ we define the \emph{distribution} at $q$ as 
\begin{equation}
\label{eq:distribution}
    \dis_q := \big\{ \xi(q,w) \, :\, w \in \R^k \big\}=\big\{X(q)\,:\,X\in\dis\big\} \subset T_qM. 
\end{equation}
This is a vector subspace of $T_qM$ whose dimension is called \emph{rank} (of the distribution) and denoted by $r(q):=\dim\dis_q\leq n$. Moreover, the distribution can be described by a family of horizontal vector fields: letting $\{e_i\}_{i=1, \dots, k}$ be the standard basis of $\R^k$, the \emph{generating family} is the set $\{X_i\}_{i=1, \dots, k}$, where 
\begin{equation}
\label{eq:NOSAUCE}
    X_i(q) := \xi(q, e_i) \qquad\forall\,q\in M, \quad \text{for }i= 1, \dots ,k.
\end{equation}
Then, according to \eqref{eq:distribution}, $\dis_q= \text{span} \{X_1(q), \dots, X_k(q)\}$. On the distribution we define the \emph{induced norm} as

\begin{equation*}
    \norm{v}_q := \inf\big\{\normc{w}_q \, :\, v = \xi(q,w) \big\} \qquad \text{for every }v\in \dis_q.
\end{equation*}
Since the infimum is actually a minimum, the function $\normdot_q$ is a norm on $\dis_q$, so that $(\dis_q, \normdot_q)$ is a normed space. Moreover, the dependence on the base point $M\ni q\mapsto\normdot_q$ is smooth.

A continuous curve $\gamma: [0,1]\to M$ is \emph{admissible} if its velocity $\dot\gamma(t)$ is defined almost everywhere and there exists a function $u=(u_1,\ldots,u_k)\in L^2([0,1];\R^k)$ such that
\begin{equation}
\label{eq:admissible_curve}
    \dot\gamma(t)=\sum_{i=1}^k u_i(t)X_i(\gamma(t)),\qquad\text{for a.e. }t\in [0,1].
\end{equation}
The function $u$ is called \emph{control}. Furthermore, given an admissible curve $\gamma$, there exists $\bar u=(\bar u_1,\dots, \bar u_k): [0,1]\to \R^k$ such that 
\begin{equation}
\label{eq:minimal_control}
    \dot\gamma(t) = \sum_{i=1}^k \bar u_i(t) X_i(\gamma(t)),
\qquad\text{and}\qquad
    \norm{\dot\gamma(t)}_{\gamma(t)} = \normc{\bar u(t)}_{\gamma(t)},\qquad\text{for a.e. }t\in[0,1].
\end{equation}
The function $\bar u$ is called \emph{minimal control}, and the map $t\mapsto \normc{\bar u(t)}_{\gamma(t)}$ belongs to $L^\infty([0,1])$, cf. \cite[Lem. 3.12]{ABB-srgeom}. We define the \emph{length} of an admissible curve as
\begin{equation}
    \ell(\gamma):=\int_0^1 \norm{\dot\gamma(t)}_{\gamma(t)} \de t\in[0,\infty).
\end{equation}
Finally, for every pair of points $q_0,q_1\in M$, we define the \emph{\sF distance} between them as
\begin{equation}
\label{eq:distance}
    \di (q_0,q_1)= \inf \left\{\ell(\gamma)\, :\, \gamma \text{ admissible, } \gamma(0)=q_0 \text{ and }\gamma(1)=q_1\right\}.
\end{equation}
Since every norm on $\R^k$ is equivalent to the standard scalar product on $\R^k$, it follows that the \sr structure on $M$ induces a distance which is locally equivalent to $\di$. Namely, denoting by $\di_{SR}$ the induced \sr distance, for every compact $K\subset M$ there exist constants $C>c>0$ (depending on $K$) such that
\begin{equation}
\label{eq:equivalent_sr_distance}
    c\,\di_{SR}\leq \di \leq C\di_{SR},\qquad\text{on }K\times K.
\end{equation}
Thus, as a consequence of the classical Filippov and Chow--Rashevskii theorems in sub-Rieman\-nian geometry, 
we obtain the following result.
\begin{prop}
\label{prop:mrchow}
    Let $M$ be a \sF manifold. The \sF distance is finite, continuous on $M\times M$ and the induced topology is the manifold one. In particular, $M$ is locally compact.
\end{prop}


\noindent We say that a \sF manifold $(M,\xi, \normdotc)$ is \emph{complete} if the metric space $(M,\di)$ is complete.

\section{The nilpotent approximation of a sub-Finsler manifold}
\label{sec:nilpotent}

We present the notion of nilpotent approximation of a sub-Finsler manifold, see \cite{MR3308372,MR1421822} for details. This construction is analogous to the nilpotent approximation of a \sr manifold, taking into account the necessary adaptations to include norms on the distribution. 

\paragraph{The flag of the distribution.}
Let $M$ be an $n$-dimensional \sF manifold with distribution $\dis$. We define the \emph{flag} of $\dis$ as the sequence of subsheafs $\dis^j\subset TM$ such that
\begin{equation}
\dis^1=\dis,\qquad \dis^{j+1}=\dis^j+[\dis,\dis^j],\qquad\forall\,j\geq 1,
\end{equation}
with the convention that $\dis^0=\{0\}$. Under the H\"ormander condition, the flag of the distribution defines an exhaustion of $T_qM$, for any point $q\in M$, i.e. there exists $s(q)\in\N$ such that:
\begin{equation}
\label{eq:sr_flag}
\{0\}=\dis^0_q\subset\dis^1_q\subset \ldots\subset \dis^{s(q)-1}_q\subsetneq \dis^{s(q)}_q=T_qM.
\end{equation}
The number $s(q)$ is called \emph{degree of nonholonomy} at $q$. We set $n_j(q) = \dim \dis^j_q$, for any $j\geq 0$, then the collection of $s(q)$ integers 
\begin{equation}
\left(n_1(q),\ldots,n_{s(q)}(q)\right)
\end{equation}
is called \emph{growth vector} at $q$, and we have $n_{s(q)}(q)=n=\dim M$. Associated with the growth vector, we can define the \emph{weights} $w_i(q)$ of $\dis$ at $q$: for every $i\in\{1,\ldots,n\}$, we set 
\begin{equation}
\label{eq:sr_weights}
w_i(q):=j,\qquad\text{if and only if}\qquad n_{j-1}(q)+1\leq i\leq n_j(q).
\end{equation}
A point $q\in M$ is said to be \emph{regular} if the growth vector is constant in a neighborhood of $q$, and \emph{singular} otherwise. The \sF structure on $M$ is said to be \emph{equiregular} if all points of $M$ are regular. In this case, the weights are constant {\red as well} on $M$ {\blue as well}. Finally, given any $q\in M$, we define the \emph{homogeneous dimension} of $M$ at $q$ as 
\begin{equation}
\label{eq:homogoeneous_dim}
\mathcal{Q}(q) = \sum_{i=1}^{s(q)
} i ( n_i(q)-n_{i-1}(q) ) =\sum_{i=1}^n w_i(q).
\end{equation}
We recall that, if $q$ is regular, then $\mathcal{Q}(q)$ coincides with the Hausdorff dimension of $(M,\di)$ at $q$, cf.\ \cite{MR806700}. Moreover, $\mathcal{Q}(q)>n$, for any $q\in M$ such that $\dis_q\subsetneq T_qM$. 

\paragraph{Privileged coordinates.}
Let $M$ be a \sF manifold with generating family $\{X_i\}_{i=1,\ldots,k}$ and $f$ be a smooth function at $q\in M$. We call \emph{nonholonomic derivative} of order $\ell\in \N$ of $f$ at $q$ the quantity
\begin{equation}
X_{j_1}\cdots X_{j_\ell}f(q),
\end{equation}
for any family of indexes $\{j_1,\ldots,j_\ell\}\subset\{1,\ldots,k\}$. Then, the \emph{nonholonomic order} of $f$ at $q$ is 
\begin{equation}
\mathrm{ord}_q(f)=\min\left\{\ell\in\N : \exists\{j_1,\ldots,j_\ell\}\subset\{1,\ldots,k\}\text{ s.t. }X_{j_1}\cdots X_{j_\ell}f(q)\neq 0\right\}.
\end{equation}

\begin{definition}[Privileged coordinates]
Let $M$ be a $n$-dimensional \sF manifold and $q\in M$. A system of local coordinates $(z_1,\ldots,z_n)$ centered at $q$ is said to be \emph{privileged} at $q$ if 
\begin{equation}
\mathrm{ord}_q(z_j) = w_j(q),\qquad\forall\, j=1,\ldots,n.
\end{equation}
\end{definition}

\noindent We remark that privileged coordinates always exist, see \cite[Sec.\ 2.1.2]{MR3308372} for the construction.

\paragraph{Nilpotent approximation.}
Let $M$ be a \sF manifold and let $q\in M$ with weights as in \eqref{eq:sr_weights}. Consider $\widehat\psi^q=(z_1,\ldots,z_n)\colon U\rightarrow V$ a chart of privileged coordinates centered at $q$, where $U\subset M$ is a relatively compact neighborhood of $q$ and $V\subset\R^n$ is a neighborhood of $0$. Then, for any $\varepsilon\in\R$, we can define the \emph{dilation} at $q$ as 
\begin{equation}
\label{eq:anis_dil}
\delta_\varepsilon\colon\R^n\rightarrow\R^n;\qquad\delta_\varepsilon
(z)=\left(\varepsilon^{w_1(q)}z_1,\ldots,\varepsilon^{w_n(q)}z_n\right).
\end{equation} 
Using such dilations, we can obtain the nilpotent (or first-order) approximation of the generating family $\{X_i\}_{i=1,\ldots,k}$. Indeed, for any $i=1\ldots,k$ we set $Y_i:=\widehat\psi^q_*X_i$ and define the vector field 
\begin{equation}
\label{eq:nilpot_approx}
 Y_i^\varepsilon:=\varepsilon\delta_{\frac{1}{\varepsilon}*}(Y_i),\qquad \forall\,\varepsilon>0.
\end{equation}
Observe that $Y_i^\varepsilon$ is defined on the set $\delta_{1/\varepsilon}(V)$ and, when $\varepsilon\to 0$, this set tends to cover the whole $\R^n$. Then, according to \cite[Lem. 10.58]{ABB-srgeom}, for every $i=1\ldots,k$ the family $\{Y_i^\varepsilon\}_{\varepsilon>0}$ admits a limit in the $C^\infty$-topology of $\R^n$ as $\varepsilon\to 0$.
In particular, for every $i=1\ldots,k$, we can define the vector field 
\begin{equation}\label{eq:convVF}
    \widehat X_i^q :=\lim_{\varepsilon\to0} Y_i^\varepsilon \in\Gamma(T\R^n).
\end{equation}

\begin{theorem}
\label{thm:nilp_approx}
The family $\{\widehat X_1^q,\ldots,\widehat X_k^q\}$ of vector fields on $\R^n$ generates a nilpotent Lie algebra of step $s(q)=w_n(q)$ and satisfies the H\"ormander condition.
\end{theorem}

\noindent This theorem is well-known in the \sr setting, cf.\  \cite{MR3308372}. Being the result purely algebraic, it carries over to the \sF setting without modifications. Recall that a Lie algebra is said to be nilpotent of step $s$ if $s$ is the smallest integer such that all the brackets of length greater than $s$ are zero. 

According to Theorem \ref{thm:nilp_approx}, we can then define the nilpotent approximation of $(M, \xi, \normdotc)$ at $q$.

\begin{definition}[Nilpotent approximation]
\label{def:nilpotent_approx}
Let $M$ be a \sF manifold and let $q\in M$. Consider the family $\{\widehat X_1^q,\ldots,\widehat X_k^q\}$ on $\R^n$, provided by  Theorem \ref{thm:nilp_approx}. The \emph{nilpotent approximation} of $M$ at the point $q$ is the \sF structure $(\widehat\xi^q, \normdotc_q)$ on $\R^n$ where
\begin{equation}
    \widehat\xi^q: \R^n\times \R^k\to \R^{2n},\qquad \widehat\xi^q(z,e_i):=(z,\widehat X_i^q(z)), \quad \text{for }i=1,\dots,k.
\end{equation} 
\end{definition}

\noindent In the following, we denote by $\widehat\dis^q$ and by $\widehat\di^q$, the distribution and the \sF distance associated to the \sF structure $(\widehat\xi^q, \normdotc_q)$ on $\R^n$, respectively.

\begin{remark}\label{rmk:dilation_inv}
    By construction, the vector fields $\widehat X_i^q$ are homogeneous of degree $-1$ with respect to $\delta_\varepsilon$, namely 
    \begin{equation}
    \label{eq:hom_Xi}
         \widehat X_i^q=\varepsilon(\delta_{\frac{1}{\varepsilon}})_*\widehat X_i^q, \qquad\forall\,\varepsilon>0,\quad  \text{for }i=1,\dots,k.
    \end{equation}
    Moreover, by definition, the norm on $\widehat\dis^q$ is given by 
    \begin{equation}
    \label{eq:def_nilp_norm}
    \norm{v}^q_z:=\inf\left\{|u|_q:\widehat\xi^q(z,u)=\sum_{i=1}^k u_i \widehat X_i^q(z)=v\right\},\qquad\forall\,v\in\widehat\dis^q_z,\ \forall\,z\in \R^n,
\end{equation}
    and it is $1$-homogeneous with respect to $\delta_\varepsilon$, meaning that $\delta_\varepsilon^*\norm{\cdot}^q_z=\varepsilon\norm{\cdot}^q_z$. Thus, the distance $\widehat\di^q$ of the nilpotent approximation is $1$-homogeneous with respect to $\delta_\varepsilon$. In addition, since $(\widehat\xi^q, \normdotc_q)$ induces a locally compact topology, by homogeneity, $(\R^n,\widehat\di^q)$ is complete.
\end{remark}

\begin{theorem}
\label{thm:francomanca}
    The nilpotent approximation of $M$ at a point $q$ is isometric to the quotient of a \sF Carnot group $G^q/H^q$, for some normal subgroup $H^q\lhd G^q$, equipped with a left-invariant norm. If $q$ is regular, then $
H^q=\{0\}$ and the nilpotent approximation is isometric to a \sF Carnot group.    
\end{theorem}

\noindent The proof of the previous theorem follows the blueprint of \cite[Thm.\ 5.21]{MR1421822}. The main difference lies of course in the norm, which is left-invariant, since the generating family is so (cf. \eqref{eq:def_nilp_norm}), and thus the metric structure of the nilpotent approximation is indeed the one of a (quotient of a) \sF Carnot group. 

\paragraph{Ball-Box theorem.} We record here an important result that can be stated using privileged coordinates. Its proof follows the equivalence between distances \eqref{eq:equivalent_sr_distance} and the corresponding \sr result, cf. \cite[Thm.\ 10.67]{ABB-srgeom}. 

\begin{theorem}[Ball-Box]
\label{thm:ball-box}
    Let $M$ be a complete \sF manifold, let $q\in M$ and let $\widehat\psi^q:U\to \R^n$ be a chart of privileged coordinates at $q$. Then, there exist constants $C_q\geq 1$ and $r_q>0$ such that
    \begin{equation}
        \text{Box}^q\left(\frac1{C_q}r\right) \subset \widehat\psi^q(B^\di(q,r))\subset \text{Box}^q(C_qr),\qquad\forall\, r\in (0,r_q)
    \end{equation}
    where $\text{Box}^q(r):=\{(z_1,\ldots,z_n) : |z_i|\leq r^{w_i(q)}, 1 \leq i \leq n\}$.
\end{theorem}

\begin{remark}
\label{rmk:uniform_ball-box}
    If $M$ is equiregular, the constants in the previous theorem can be chosen uniformly on compact sets. Namely, for every $K\subset M$ compact, the constants $C_q$ and $r_q$ of Theorem \ref{thm:ball-box} can be taken to be independent of $q\in K$, cf. \cite[Thm.\ 2.3]{MR3308372}.
\end{remark}

\section{The metric measure tangent to a \sF manifold}
\label{sec:tangent}

In the present section, we show that the nilpotent approximation of a complete \sF manifold $M$ at $q$, equipped with (a scalar multiple of) the Lebesgue measure, is its metric measure tangent. The metric side of the result is contained in \cite{MR4645068}. Below, we give a self-contained construction of the metric tangent. 

\subsection{The approximating family of \sF structures}


Let $M$ be a complete \sF manifold and let $q\in M$. Fix a chart of privileged coordinates centered at $q$, $(z_1,\ldots,z_n):U\to V\subset\R^n$. With a slight abuse of notation, we identify $M$ with $\R^n$ and work in coordinates. For every $\varepsilon>0$, we define the \sF structure $(\xi^\varepsilon, \normdotc^\varepsilon)$ on $\delta_{1/\varepsilon}(V)$ as 
\begin{align}
    \xi^\varepsilon&: \delta_{1/\varepsilon}(V)\times \R^{k}\to \R^{2n},\qquad \xi^\varepsilon(z,e_i):=(z,Y_i^\varepsilon(z))\qquad \text{for }i=1,\dots k.\\[12pt]
    \normdotc^\varepsilon &:\delta_{1/\varepsilon}(V)\times \R^{k}\to \R, \qquad\qquad \,|u|^\varepsilon_z := |(u_1,\dots,u_k)|_{\delta_\varepsilon(z)},
\end{align}
where the vector fields $Y_i^\varepsilon$ are defined in \eqref{eq:nilpot_approx}. For every $\varepsilon>0$, the length-minimization procedure \eqref{eq:distance} defines a distance on $\delta_{1/\varepsilon}(V)$, which we denote by $\di^\varepsilon$. Note that $\di^1\equiv\di$ on $V\times V$.

A first important property of the approximating structure is a kind of homogeneity with respect to the anisotropic dilations. 

\begin{lemma}
\label{lem:scaling}
    For every $\varepsilon>0$ it holds that 
    \begin{equation}\label{eq:scaling}
        \di^\epsilon\big(\delta_{1/\varepsilon}(x),\delta_{1/\varepsilon}(y)\big)= \frac{1}{\varepsilon} \di (x, y), \qquad \text{for every }x, y \in V.
    \end{equation}
\end{lemma}

\begin{proof}
    By definition of $Y_i^\varepsilon$, cf.\ \eqref{eq:nilpot_approx}, we have the following: 
    \begin{equation}
    \label{eq:hom_Yi}
        \rho(\delta_{\frac1\rho})_*Y_i^\varepsilon=Y_i^{\rho\varepsilon},\qquad\forall\,\varepsilon,\rho>0.
    \end{equation}    
    In addition, by our definition of norm, for every $z\in \delta_{1/\varepsilon}(V)$ and $u\in\R^{2k}$, we have
    \begin{equation}
    \label{eq:hom_norm}
        |u|^\varepsilon_{\delta_\rho(z)} =|(u_1,\dots,u_k)|_{\delta_{\rho\varepsilon}(z)}.
    \end{equation}
    From \eqref{eq:hom_Yi} and \eqref{eq:hom_norm}, we immediately deduce that, for every $\varepsilon,\rho>0$: 
    \begin{equation}
    \delta_\rho^*\di^\varepsilon=\rho\,\di^{\rho\varepsilon}, \qquad \text{on }\delta_{\frac1{\rho\varepsilon}}(V)\times\delta_{\frac1{\rho\varepsilon}}(V).
    \end{equation}
    The claim follows choosing $\rho=1/\varepsilon$.
\end{proof}

The next result is crucial in characterizing the metric tangent of a \sF manifold. We refer to \cite[Thm.\ 1.4(iv)]{MR4645068} for details. (See also \cite[Thm.\ C.2]{MR4645068} for a simplified proof in the case of smooth structures.)

\begin{theorem}\label{thm:convergencetoNA}
    The family of distances $\{\di^\varepsilon\}_{\varepsilon>0}$ converges uniformly on compact sets of $\R^n \times \R^n$ to the distance $\widehat \di^q$, as $\varepsilon\to 0$.
\end{theorem}

\begin{proof}
    We check the hypothesis of \cite[Thm.\ 1.4(iv)]{MR4645068}. Firstly, the vector field $Y_i^\varepsilon$ converges to $\widehat X_i^q$ locally uniformly on $\R^n$, for every $i$, as $\varepsilon\to 0$, cf.\ \eqref{eq:convVF}. Secondly, the norm $\normdotc^\varepsilon$ converges to $\normdotc_q$ locally uniformly as $\varepsilon\to 0$ (as functions on $\R^n\times\R^k$). Thirdly, the limit \sF structure $(\widehat\xi^q,\normdotc_q)$ defines a complete distance on $\R^n$. 
\end{proof}

\subsection{Bounded measures on a sub-Finsler manifold}

In this section, we describe a class of measures whose blow-up is the Lebesgue measure almost-everywhere, up to a constant. 

\begin{definition}[Bounded measures]
\label{def:bounded_measure}
    Let $M$ be a $n$-dimensional \sF manifold, equipped with a Borel measure $\m$. We say that $\m$ is a \emph{bounded measure} if, for every coordinate chart $\psi:U\to \R^n$, $\psi_\#\m\ll\Leb^n$ and 
    \begin{equation}
    \label{eq:boundedness}
        \psi_\#\m=\rho\Leb^n,\qquad\text{with}\qquad c\leq \rho\leq C\quad\Leb^n\text{-a.e.},
    \end{equation}
    for some positive constants $C\geq c>0$, depending on the chart.
\end{definition}

\begin{lemma}
\label{lem:density}
    Let $M$ be a complete $n$-dimensional sub-Finsler manifold, equipped with a bound\-ed measure $\m$. Then, for every $q\in M$, there exists $C_q\geq 1$ and $r_q>0$ such that
    \begin{equation}
         C_q^{-1}r^{\mathcal Q(q)}\leq \m(B(q,r))\leq C_q r^{\mathcal Q(q)},\qquad\forall\,q\in M,\ r\in (0,r_q),
    \end{equation} 
    where $\mathcal Q(q)$ is the homogeneous dimension of $M$ at $q$. 
\end{lemma}

\begin{proof}
    Let $q\in M$ and let $\widehat\psi^q:U\to \R^n$ be a set of privileged coordinates centered at $q$ and consider $\widehat\psi^q_\#\m$. On the one hand, by assumptions, there exists a measurable function $\rho:\R^n\to \R_+$ satisfying \eqref{eq:boundedness}. Then, we easily see that, for $r>0$ sufficiently small,
    \begin{equation}
    \label{eq:previousprevious_inequality}
       c\, (\widehat\psi^q)_\#^{-1}\Leb^n(B(q,r)) \leq \m(B(q,r))\leq C\,(\widehat\psi^q)_\#^{-1}\Leb^n(B(q,r)). 
    \end{equation}
    On the other hand, from the Ball-Box theorem, cf.\ Theorem \ref{thm:ball-box}, we deduce that there exist constants $\tilde C_q$ and $r_q\in (0,\diam U)$, such that
    \begin{equation}
        \tilde C_q^{-1} r^{\mathcal Q(q)}\leq (\widehat\psi^q)^{-1}_\#\Leb^n(B(q,r))\leq \tilde C_qr^{\mathcal Q(q)},\qquad\forall\,r\leq r_q.
    \end{equation}
    The combination of the previous inequality and \eqref{eq:previousprevious_inequality} concludes the proof.
\end{proof}

\begin{definition}[Lebesgue points for $\m$]
\label{def:lebesgue_points}
    Let $M$ be a \sF manifold, equipped with a bounded measure $\m$. Then, we say that $q\in M$ is a \emph{Lebesgue point for} $\m$ if there exists a coordinate chart $\psi:U\to\R^n$, defined on a neighborhood $U$ of $q$, such that
    \begin{equation}
    \label{eq:limit_density}
        \lim_{r\to 0}\frac{1}{r^{\mathcal Q(q)}}\int_{\psi(B(q,r))}|\rho-\rho(\psi(q))|\de\Leb^n=0,
    \end{equation}
    where $\psi_\#\m=\rho\Leb^n$ and $\mathcal Q(q)$ is the homogeneous dimension of $M$ at $q$. 
\end{definition}

\begin{lemma}
\label{lem:boh}
    Let $M$ be a complete \sF manifold, equipped with a bounded measure $\m$. Then, the definition of Lebesgue point for $\m$ is independent on the choice of $\psi$ and $\m$-a.e.\ point of $M$ is a Lebesgue point for $\m$, with strictly positive density.
\end{lemma}

\begin{proof}
    Let $\psi:U\subset M\to \R^n$ be a coordinate chart, with $U$ relatively compact. First of all, we show that $\Leb^n$  satisfies the asymptotic doubling property on $(\psi(U),(\psi^{-1})^*\di)$, namely 
    \begin{equation}
    \label{eq:asymp_doubling}
        \limsup_{r\to 0}\frac{\Leb^n(\psi(B(q,2r)))}{\Leb^n(\psi(B(q,r)))}<\infty, 
    \end{equation}
    for all $q\in U$. Indeed, by the Ball-Box theorem, cf.\ Theorem \ref{thm:ball-box}, there exist constants $C_q\geq 1$ and $r_q>0$ such that 
    \begin{equation}
    \label{eq:comparison_measure}
        \frac{1}{C_q}r^{\mathcal Q(q)}\leq \Leb^n(\widehat\psi^q(B(q,r)))\leq C_q r^{\mathcal Q(q)},\qquad\forall r<r_q,
    \end{equation}
    where $\widehat\psi^q:V_q\subset U\to \R^n$ is a chart of privileged coordinates at $q$. Thus, for every $r< r_q$, we have: 
    \begin{align}
        \Leb^n(\psi(B(q,2r)))&=\int_{\widehat\psi^q(B(q,2r))}|\det d(\psi\circ (\widehat\psi^q)^{-1})|\de\Leb^n\\ 
        &\leq \sup_{\widehat \psi^q(V_q)}|\det d(\psi\circ (\widehat\psi^q)^{-1})|\,\Leb^n(\widehat\psi^q(B(q,2r)))\\
        &\leq \sup_{\widehat\psi^q(V_q)}|\det d(\psi\circ (\widehat\psi^q)^{-1})|\,  C_q(2r) ^{\mathcal Q(q)} \\
        &\leq \sup_{\widehat\psi^q(V_q)}|\det d(\psi\circ (\widehat\psi^q)^{-1})|\, C_q^2 2^{\mathcal Q(q)}\Leb^n(\widehat\psi^q(B(q,r)))\\
        &\leq \sup_{\widehat\psi^q(V_q)}|\det d(\psi\circ (\widehat\psi^q)^{-1})|\,\sup_{\psi(U)}|\det d(\widehat\psi^q\circ \psi^{-1})|\, C_q^2 2^{\mathcal Q(q)}\Leb^n(\psi(B(q,r))).
    \end{align}
    This shows that condition \eqref{eq:asymp_doubling} is fulfilled.
    Therefore, Lebesgue differentiation theorem holds, cf. \cite[Sec.\ 3.4]{MR3363168}, and we deduce that, for $\Leb^n$-a.e. $x\in\psi(U)$, we have
    \begin{equation}
        \lim_{r\to 0}\frac{1}{\Leb^n(\psi(B(p,r)))}\int_{\psi(B(p,r))}|\rho-\rho(x)|\de\Leb^n=\lim_{r\to 0}\frac{1}{r^{\mathcal Q(p)}}\int_{\psi(B(p,r))}|\rho-\rho(x)|\de\Leb^n=0,
    \end{equation}
    having set $p:=\psi^{-1}(x)$. From this, we immediately see that, if the limit \eqref{eq:limit_density} is independent on the coordinate chart, $\m$-a.e.\ point of $M$ is a Lebesgue point for $\m$. 
    
    Second of all, let $\varphi:V\to\R^n$ be another coordinate chart containing $q$. Then, $\varphi_\#\m=\eta\Leb^n$ and we need to check that \eqref{eq:limit_density} implies that
    \begin{equation}
    \label{eq:limit_density_claim}
        \lim_{r\to 0}\frac{1}{r^{\mathcal Q(q)}}\int_{\varphi(B(q,r))}|\eta-\eta(\varphi(q))|\de\Leb^n=0.
    \end{equation}
    Consider the smooth change of coordinates given by $\psi\circ\varphi^{-1}:\R^n\to\R^n$ and observe that, by definition, we have
    \begin{equation}
        (\psi\circ\varphi^{-1})_\#(\eta\Leb^n)=\rho\Leb^n,
    \end{equation}
    which implies that $\eta=\rho\circ\psi\circ\varphi^{-1}$, $\Leb^n$-a.e.. Hence, using the change of variables $x\mapsto \psi\circ\varphi^{-1}(x)$ in \eqref{eq:limit_density_claim}, we obtain: 
    \begin{equation}
        \frac{1}{r^{\mathcal Q(q)}}\int_{\varphi(B(q,r))}|\eta-\rho(\psi(q))|\de\Leb^n=\frac{1}{r^{\mathcal Q(q)}}\int_{\psi(B(q,r))}|\rho-\rho(\psi(q))|\cdot|\det d(\varphi\circ\psi^{-1})|\de\Leb^n.
    \end{equation}
    Since $\varphi\circ\psi^{-1}$ is smooth, we easily see that the right-hand side converges to $0$, as a consequence of \eqref{eq:limit_density}. Recalling that $\eta=\rho\circ\psi\circ\varphi^{-1}$, $\Leb^n$-a.e., we conclude the proof of the claim, possibly choosing a different representative of $\eta$.  
\end{proof}

When $M$ is an equiregular \sF manifold, Lemma \ref{lem:density} can be improved to a complete characterization of \sF bounded measures, in terms of the notion of $\mathcal Q$-Ahlfors regularity.

\begin{definition}
\label{def:ahlfors}
    Let $(\X,\di,\m)$ be a metric measure space and let $\mathcal Q>1$. We say that $\m$ is \emph{locally $\mathcal Q$-Ahlfors regular} if, for every compact $K\subset \X$, there exists $C\geq 1$ and $r_0\in (0,\diam\X]$, depending on $K$, such that
    \begin{equation}
         C^{-1}r^{\mathcal Q}\leq \m(B(p,r))\leq C r^{\mathcal Q},\qquad\forall\,p\in K,\ r\in (0,r_0).
    \end{equation}
\end{definition}

\begin{corollary}
\label{cor:density}
    Let $M$ be a complete equiregular \sF manifold of homogeneous dimension $\mathcal Q$. Then, a measure $\m$ on $M$ is bounded if and only if it is locally $\mathcal Q$-Ahlfors regular. 
\end{corollary}

\begin{proof}
    The only if part of the statement is the content of Lemma \ref{lem:density}, taking into account that all the constants can be chosen locally uniformly since $M$ is equiregular. To prove the converse statement, let $\psi:U\subset M\to\R^n$ be a chart and define $\mu:=\psi_\#\m$ and $\nu:=\restr{\Leb^n}{\psi(U)}$. On the one hand, from the proof of Lemma \ref{lem:boh}, we deduce that $\nu$ is asymptotically doubling and thus $(\psi(U),(\psi^{-1})^*\di,\nu)$ satisfies Vitali's covering lemma, cf.\ \cite[Thm.\ 3.4.3]{MR3363168}. On the other hand, from the upper bound of the uniform Ball-Box theorem, cf.\ Remark \ref{rmk:uniform_ball-box}, we obtain the estimate: 
    \begin{equation}
    \label{eq:theestimate}
        \overline D_\nu\mu(p):=\limsup_{r\to 0}\frac{\mu(\psi(B(p,r)))}{\nu(\psi(B(p,r)))}\leq C,\qquad\forall\,p\in U,
    \end{equation}
    for some positive constant $C$ depending only on $U$. From \eqref{eq:theestimate}, a classical application of Vitali covering's theorem ensures that for every open set $A\subset \psi(U)$, $\mu(A)\leq C\nu(A)$. Finally, the monotone class theorem allows to deduce that 
    \begin{equation}
        \mu(E)\leq C\nu(E),\qquad\forall\,E\subset\psi(U)\ \text{Borel}.
    \end{equation}
    This implies that $\mu\ll\nu$ and the density is bounded by $C$. To prove that $\nu\ll\mu$ with a bounded density we repeat the same argument, observing that also $\mu$ is asympotically doubling, thanks to the $\mathcal Q$-Ahlfors regularity, and using the lower bound of the uniform Ball-Box theorem.
\end{proof}

\begin{remark}
\label{rmk:mini_flex}
    On an equiregular \sF manifold, the $\mathcal Q$-Hausdorff measure is locally $\mathcal Q$-Ahlfors regular, as a consequence of the (uniform) Ball-Box theorem. Therefore, it is bounded in the sense of Definition \ref{def:bounded_measure}. 
\end{remark}

\subsection{Convergence to the metric measure tangent}

In this section, we prove Theorem \ref{thm:tangente}. The metric side is the content of \cite[Thm.\ 1.5]{MR4645068}. We include a proof for completeness. The blow-up of the measure has been studied in \cite{MR3388884,SRmeasures}, only for smooth measures (in the \sr case). 

Recall that a pointed metric measure space is a metric measure space with a distinguished point. For such class of spaces, a natural notion of convergence is the pointed measured Gromov-Hausdorff convergence, see \cite{GigMonSav}. 

\begin{definition}[pmGH-convergence]
\label{def:pmGH_convergence}
     Let $(\X_n, \di_n,\m_n, x_n)$, $n\in\N\cup\{\infty\}$ be a collection of pointed metric measure spaces. Then, we say that $(\X_n, \di_n,\m_n, x_n)$ converges to $(\X_\infty, \di_\infty,\m_\infty, x_\infty)$ in the pointed measured Gromov-Hausdorff sense (pmGH for short) if for any $\varepsilon,R > 0$ there exists $N(\varepsilon,R)\in\N$ such that for all $n>N(\varepsilon,R)$ there exists a Borel map $f^{\varepsilon,R}_n : B^{\di_n}(x_n,R)\to \X_\infty$ such that
     \begin{enumerate}
         \item[(i)] $f^{\varepsilon,R}_n(x_n)=x_\infty$,
         \item[(ii)] $\sup_{x,y \in B^{\di_n}(x_n,R)} |\di_n(x,y)- \di_\infty(f^{\varepsilon,R}_n(x), f^{\varepsilon,R}_n(y))|\leq \varepsilon$,
         \item[(iii)] the $\varepsilon$-neighbourhood of $f^{\varepsilon,R}_n (B^{\di_n}(x_n,R))$ contains $B^{\di_\infty}(x_\infty,R-\varepsilon)$,
         \item[(iv)] $(f^{\varepsilon,R}_n)_\#(\restr{\m_n}{B^{\di_n}(x_n,R)})$ weakly converges to $\restr{\m}{B^{\di_\infty}(x_\infty,R)}$ as $n\to\infty$, for a.e. $R > 0$.
     \end{enumerate}
\end{definition}  

 The notion of metric measured tangents for a general metric measure space can be given according to the notion of pmGH-convergence. Let $(\X,\di,\m)$ be a metric measure space, let $x\in\supp(\m)$ and $r\in(0,1)$. Define 
\begin{equation}
\label{eq:tangent_measure}
    \m_r^x:= \left(\int_{B(x,r)}1-\frac1r\di(\cdot,x)\de\m\right)^{-1}\m.
\end{equation}

 \begin{definition}[The collection of tangent spaces ${\rm Tan}(\X,\di,\m,x)$]
    Let $(\X,\di,\m)$ be a metric measure space, let $x\in\supp(\m)$ and $r\in(0,1)$. and consider the rescaled and normalized pointed metric measure space $(\X, \di_r,\m^x_r,x)$, where $\di_r:=r^{-1}\di$ and $\m_r^x$ is defined in \eqref{eq:tangent_measure}. A pointed metric measure space $({\sf Y},\di_Y, \mathfrak{n}, y)$ is called a tangent to $(\X,\di,\m)$ at $x$ if there exists a sequence of radii $r_i\to 0^+$ so that 
    $(\X, \di_{r_i},\m^x_{r_i},x) \to ({\sf Y},\di_Y, \mathfrak{n}, y)$ as $i\to\infty$ in the pointed measured Gromov-Hausdorff topology. We denote the collection of all the tangents of $(\X,\di,\m)$ at $x$ by ${\rm Tan}(\X,\di,\m,x)$.    
 \end{definition}

We are now in position to prove Theorem \ref{thm:tangente}. 

 \begin{theorem}
 \label{thm:tangents}
     Let $M$ be an $n$-dimensional complete \sF manifold, equipped with a bounded measure $\m$. Then, for $\m$-a.e.\ $q\in M$, we have 
     \begin{equation}
         {\rm Tan}(M,\di,\m,q)=\{(\widehat M^q,\widehat\di^q,m(q)\Leb^n,0)\},
     \end{equation}
     where $(\widehat M^q,\widehat\di^q)$ is the nilpotent approximation of $M$ at $q$ and is isometric to a quotient \sF Carnot group and $m(q)>0$ is defined by 
     \begin{equation}
     \label{eq:normalization_tangent_measure}
         m(q):=\left(\int_{\hat B^q(0,1)}(1-\widehat\di^q)\de\Leb^n\right)^{-1}.
     \end{equation} 
 \end{theorem}

 \begin{proof}
     Let $q\in M$ be a Lebesgue point for $\m$ and let  $\widehat\psi^q:U\to \R^n$ be a chart of privileged coordinates at $q$, with $U$ open and relatively compact. Fix $\varepsilon,R>0$ and choose $r_0=r_0(\varepsilon,R)$ sufficiently small to ensure that $B^{\di_r}(q,R)=B^{\di}(q,rR)\subset U$ for every $r<r_0$. In this way, we can work in coordinates and, for ease of notation, we can identify $M$ with $\R^n$ and $q$ with the origin (we also drop the superscript $q$ of the nilpotent approximation). Then, define 
     \begin{equation}
         f^{\varepsilon,R}_r : B^{\di_r}(0,R)\to \R^n;\qquad f^{\varepsilon,R}_r(p):=\delta_{r}^{-1}(p),
     \end{equation}
     where $\delta_r:\R^n\to\R^n$ is the dilation defined in \eqref{eq:anis_dil}. Note that $f^{\varepsilon,R}_r(0)=0$ and therefore item (i) of Definition \ref{def:pmGH_convergence} is verified. Possibly taking a smaller $r_0$, we may assume that $B^\di(0,r_0R)\subset V$. As a consequence of Lemma \ref{lem:scaling}, we have 
     \begin{equation}
     \label{eq:comparison_sets}
        \delta_{1/r}(B^\di(0,rR))=B^{\di^r}(0,R),\qquad\forall\, r\leq r_0,
    \end{equation}
    From Theorem \ref{thm:convergencetoNA}, we immediately deduce that $\delta_{1/r}(B^\di(0,rR))\to B^{\widehat \di}(0,R)$ as $r\to 0$ in the Hausdorff topology, proving that also item (iii) of Definition \ref{def:pmGH_convergence} is verified. Combining this convergence with Lemma \ref{lem:scaling} and Theorem \ref{thm:convergencetoNA}, we obtain
    \begin{align*}
         \sup_{p,p' \in B^{\di_r}(0,R)} \Big|\di_r(p,p')- \widehat\di(f^{\varepsilon,R}_r(p), f^{\varepsilon,R}_r(p'))\Big| &=\sup_{p,p' \in B^{\di}(0,rR)} \Big|\frac1r\di(p,p')-\widehat\di\big(\delta_{\frac1r}(p),\delta_{\frac1r}(p')\big)\Big|\\
         &=\sup_{p,p' \in B^{\di}(0,rR)} \Big|\di^r\big(\delta_{\frac1r}(p),\delta_{\frac1r}(p')\big)- \widehat\di\big(\delta_{\frac1r}(p),\delta_{\frac1r}(p')\big)\Big|\\
         &=\sup_{x,x' \in \delta_{1/r}(B^{\di}(0,rR))} \big|\di^r(x,x')- \widehat\di(x,x')\big|\xrightarrow{r\to0} 0.
     \end{align*}
     Hence, we have proved that, up to possibly choosing a smaller $r_0$, item (ii) of Definition \ref{def:pmGH_convergence} is verified. Therefore we are left to check item (iv). 
     
     First of all, $\m$ is a bounded measure, hence $\widehat\psi^q_\#\m=\rho\Leb^n$ for a measurable function $\rho$ with $0<c\leq \rho\leq C$. Thus, since $0\in\R^n$ is a Lebesgue point for $\m$, after the change of variables $x\mapsto\delta^{-1}_r(x)$, we have 
     \begin{equation}
     \label{eq:density_point_change_of_var}
         0=\lim_{r\to 0}\frac{1}{r^{\mathcal Q(q)}} \int_{B^{\di}(0,r)}|\rho-\rho(0)|\de\Leb^n=\lim_{r\to 0}\int_{\delta^{-1}_r(B^{\di}(0,r))}|\rho\circ\delta_r-\rho(0)|\de\Leb^n,
     \end{equation}
     where $\rho(0)$ is strictly positive and finite. Second of all, for every $\varphi\in L^\infty(\R^n)$, we claim that 
     \begin{equation}
     \label{eq:final_claim}
         \lim_{r\to 0} \int_{\delta^{-1}_r(B^{\di}(0,rR))}\varphi(\rho\circ\delta_r)\de\Leb^n=\rho(0)\int_{ B^{\widehat\di}(0,R)}\varphi\de\Leb^n.
     \end{equation}
     Indeed, recall that $\delta^{-1}_r(B^\di(0,rR))\to B^{\widehat\di}(0,R) $ as $r\to 0$, as sets and we can compute: 
    \begin{multline}
        \left|\int_{\delta^{-1}_r(B^{\di}(0,rR))}\varphi(\rho\circ\delta_r)\de\Leb^n-\rho(0)\int_{ B^{ \widehat\di}(0,R)}\varphi\de\Leb^n\right|
        \\
        \leq\int_{\delta^{-1}_r(B^\di(0,rR))}\left|\varphi(\rho\circ\delta_r)-\rho(0)\varphi\right|\de\Leb^n+\rho(0)\left|\int_{\delta^{-1}_r(B^\di(0,rR))}\varphi\de\Leb^n-\int_{B^{\widehat\di}(0,R)}\varphi\de\Leb^n\right|,
    \end{multline}
    which converges to $0$ as a consequence of \eqref{eq:comparison_sets} and \eqref{eq:density_point_change_of_var}.
    We are in position to prove the weak convergence of the sequence $(f^{\varepsilon,R}_r)_\#\restr{\m_r^0}{B^{\di_r}(0,R)}$ to the tangent measure $\restr{ m(q)\Leb^n}{B^{\widehat\di}(0,R)}$.
    Let $\varphi\in C_c(\R^n)$, then by definition of $\m_r^0$, we have
     \begin{equation}
     \begin{split}
         \int_{\R^n} \varphi \de (f^{\varepsilon,R}_r)_\#(\restr{\m_r^0}{B^{\di_r}(0,R)})&=\left(\int_{B^{\di}(0,r)}1-\frac1r\di(\cdot,0)\de\m\right)^{-1}\int_{B^{\di_r}(0,R)}(\varphi\circ \delta_r^{-1})\rho\de\Leb^n \\
         &=\left(\int_{B^{\di}(0,r)}1-\di_r(\cdot,0)\de\m\right)^{-1}r^{\mathcal Q(q)}\int_{\delta_r^{-1}(B^{\di}(0,rR))}\varphi(\rho\circ\delta_r)\de\Leb^n,
    \end{split}
     \end{equation}
     having performed a change of variable $x\mapsto\delta_r(x)$. For the renormalization term, using a change of variable and \eqref{eq:final_claim}, we have  
 \begin{equation}
     \begin{split}
    \label{eq:aux_convergence}
         \int_{B^{\di}(0,r)}1-\di_r(\cdot,0)\de\m &=r^{\mathcal Q(q)}\int_{\delta^{-1}_r(B^{\di}(0,r))} (1-\di_r(\delta_r(\cdot),0))\rho\circ\delta_r\de\Leb^n \\ 
         &=r^{\mathcal Q(q)}\int_{\delta^{-1}_r(B^{\di}(0,r))} (1-\di^r(\cdot,0))\rho\circ\delta_r\de\Leb^n \stackrel{r\to0}{\sim}
 m(q)^{-1}\,\rho(0) r^{\mathcal Q(q)},
     \end{split}
     \end{equation}
     where $m(q)>0$ is defined in \eqref{eq:normalization_tangent_measure}. Hence, combining the convergence \eqref{eq:aux_convergence} with \eqref{eq:final_claim}, we may deduce that 
     \begin{equation}
     \begin{split}
         \lim_{r\to0}\int_{\R^n} \varphi \de (f^{\varepsilon,R}_r)_\#(\restr{\m_r^0}{B^{\di_r}(0,R)})&=m(q)\int_{B^{\widehat\di}(0,R)}\varphi\de\Leb^n=m(q)\int_{\R^n}\varphi\de(\restr{\Leb^n}{B^{\widehat\di}(0,R)}),
    \end{split}
     \end{equation}
     concluding the proof of item (iv).
 \end{proof}

In light of Remark \ref{rmk:mini_flex}, we obtain a straightforward corollary of the previous result. 

\begin{corollary}
\label{cor:straightforward}
     Let $M$ be a $n$-dimensional, complete and equiregular \sF manifold and let $\mathcal Q$ be its homogeneous dimension. Then, for $\Haus^\mathcal{Q}$-a.e.\ $q\in M$, we have 
     \begin{equation}
         {\rm Tan}(M,\di,\Haus^\mathcal{Q},q)=\{(\widehat M^q,\widehat\di^q,m(q)\Leb^n,0)\},
     \end{equation}
     where $(\widehat M^q,\widehat\di^q)$ is the nilpotent approximation of $M$ at $q$ and is isometric to a \sF Carnot group and $m(q)>0$ is defined by \eqref{eq:normalization_tangent_measure}.
\end{corollary}

\begin{remark}
    For a \sF Carnot group of topological dimension $n$ and homogeneous dimension $\mathcal Q$, the Lebesgue measure $\Leb^n$ and the Hausdorff measure $\Haus^\mathcal{Q}$ are bi-invariant Haar measures, thus they coincide up to a scalar multiple. Therefore, Corollary \ref{cor:straightforward} says that
    \begin{equation}
        {\rm Tan}(M,\di,\Haus_\di^\mathcal{Q},q)=\{(\widehat M^q,\widehat\di^q,c\Haus_{\widehat\di^q}^\mathcal{Q},0)\},\qquad\text{for } \Haus_\di^\mathcal{Q}\text{-a.e.\ } q\in M.
    \end{equation}
\end{remark}

\begin{remark}
    Given a \sF manifold $M$ and a regular point $p\in M$, the previous result holds in a neighborhood $U$ of $p$ (where $\mathcal Q$ is replaced by the Hausdorff dimension of $U$). 
\end{remark}

\section{Applications to the curvature-dimension condition}
\label{sec:contact}

Recall that the Heisenberg group $\hei$ is the only Carnot group of topological dimension $3$ and Hausdorff dimension $4$. A \sF Heisenberg group is the Heisenberg group, equipped with a norm $\normdotc:\R^2\to\R_+$ on its first layer. We refer the reader to \cite[Sec.\ 3]{borza2024measure} for a precise definition. The next results deal with 3D-contact \sF manifolds, that are equiregular \sF manifolds with growth vector $(2,3)$. See \cite{ABB-srgeom} for details. We stress that we do not make any assumptions on the norm. 

\begin{prop}\label{prop:mancafranco}
    Let $(M,\xi,\normdotc)$ be a complete 3D-contact \sF manifold, equipped with a locally $4$-Ahlfors regular measure $\m$. Then, for $\m$-a.e. $q\in M$, we have 
     \begin{equation}
         {\rm Tan}(M,\di,\m,q)=\{(\hei,\widehat\di^q,m(q)\Leb^3,0)\},
     \end{equation}
     where $\hei$ is the Heisenberg group, $\widehat\di^q$ is the \sF distance on $\hei$ induced by $\normdotc_q:\R^2\to\R_+$, and $m(q)>0$ is defined by \eqref{eq:normalization_tangent_measure}. 
\end{prop}

\begin{proof}
    Applying Theorem \ref{thm:tangents}, we deduce that, at $\m$-almost every point $q\in M$, the nilpotent approximation $(\widehat M^q,\widehat\di^q)$ of $M$, equipped with the measure $m(q)\Leb^3$ and with distinguished point $0$, is the unique (pointed) metric measure tangent at $q$. Since $q$ is a regular point, Theorem \ref{thm:francomanca} implies that $(\widehat M^q,\widehat\di^q)$ is a Carnot group of (topological) dimension $3$ and thus it must be the \sF Heisenberg group, where $\widehat\di^q$ is defined in Definition \ref{def:nilpotent_approx}.    
\end{proof}

\begin{prop}
\label{prop:brodo}
    Let $(M,\xi,\normdotc)$ be a complete 3D-contact \sF manifold, equipped with a locally $4$-Ahlfors regular measure $\m$. Then, $(M,\di,\m)$ does not satisfy the $\cd(K,N)$ condition for every $K\in\R$ and $N\in (1,\infty)$.
\end{prop}

\begin{proof}
    Arguing by contradiction, we assume that $(M,\di,\m)$ satisfies the $\cd(K,N)$ condition for some $K\in\R$ and $N\in (1,\infty)$.
    The scaling property of the $\cd$ condition (see \cite[Prop.\ 1.4]{sturm2006}) guarantees that the metric measure space $(\X, \di_r,\m^x_r)$ satisfies the $\cd(r^2 K, N)$ condition. Moreover, as the $\cd$ condition is stable with respect to the pointed measured Gromov-Hausdorff convergence (see \cite{GigMonSav}), Proposition \ref{prop:mancafranco} allows us to conclude that 
    \begin{equation}
        (\hei,\widehat\di^q,m(q)\Leb^3,0) \quad \text{is a }\cd(0,N)\text{ space.} 
    \end{equation}
    Finally, Theorem 1.6 in \cite{borza2024measure} gives the desired contradiction.
\end{proof}

\begin{remark}
    In certain situations, the strategy used in the previous proof can be refined to show the failure of the measure contraction property $\MCP(K,N)$ (cf. \cite{ohta2007} for a definition). Indeed, under the same assumptions of Proposition \ref{prop:brodo}, if there exists a positive $\m$-measure set where $\normdotc_q$ is neither $C^1$ nor strongly convex, then Theorem 1.1 of \cite{borza2024measure} implies that $\MCP(K,N)$ fails for every $K\in\R$ and every $N\in (1,\infty)$. 
\end{remark}


The following corollary is a specification of Proposition \ref{prop:brodo} in the case of the Heisenberg group. We believe it is worth emphasize it, in light of \cite[Conj.\ 1.3]{magnabosco2023failure}.

\begin{corollary}
    \label{cor:brodoriscaldato}
    Let $\hei$ be a \sF Heisenberg group, equipped with a locally $4$-Ahlfors regular measure $\m$. Then, $(\hei,\di,\m)$ does not satisfy the $\cd(K,N)$ condition for every $K\in\R$ and $N\in (1,\infty)$.
\end{corollary}

Finally, we extend the known \sr results to bounded measures.

\begin{prop}
    Let $M$ be a complete \sr manifold with $\dim M=n$, equipped with a bounded measure $\m$. Assume there exists an $\m$-positive set $\mathcal S\subset M$ such that $r(q)<n$, for every $q\in \mathcal S$. Then, $(M,\di,\m)$ does not satisfy the $\cd(K,N)$ condition for every $K\in\R$ and $N\in (1,\infty)$.
\end{prop}

\begin{proof}
    Assume by contradiction that $(M,\di,\m)$ satisfies the $\cd(K,N)$ condition for some $K\in\R$ and $N\in (1,\infty)$. Since $\mathcal{S}$ has positive measure, we can find $q\in \mathcal{S}$ which is a Lebesgue point for $\m$, cf.\ Lemma \ref{lem:boh}. By Theorem \ref{thm:nilp_approx}, the unique (pointed) metric measure tangent of $M$ at $q$ is a \sr manifold $(\widehat M^q,\widehat\di^q)$, equipped with the measure $m(q)\Leb^n$ and with distinguished point $0$. 
    Arguing as in the proof of Proposition \ref{prop:brodo}, we conclude that $(\widehat M^q,\widehat\di^q, m(q)\Leb^n)$ is a $\cd(0,N)$ space. This is in contradiction with \cite[Thm.\ 1.2]{MR4623954}, since $r(0)<n$. 
\end{proof}

\bibliographystyle{alpha} 
\bibliography{bibliography}

\begin{thebibliography}{ALDNG23}

\bibitem[ABB20]{ABB-srgeom}
Andrei Agrachev, Davide Barilari, and Ugo Boscain.
\newblock {\em A comprehensive introduction to sub-{R}iemannian geometry}, volume 181 of {\em Cambridge Studies in Advanced Mathematics}.
\newblock Cambridge University Press, Cambridge, 2020.
\newblock From the Hamiltonian viewpoint, With an appendix by Igor Zelenko.

\bibitem[ALDNG23]{MR4645068}
Gioacchino Antonelli, Enrico Le~Donne, and Sebastiano Nicolussi~Golo.
\newblock Lipschitz {C}arnot-{C}arath\'{e}odory structures and their limits.
\newblock {\em J. Dyn. Control Syst.}, 29(3):805--854, 2023.

\bibitem[Bel96]{MR1421822}
Andr\'{e} Bella\"{\i}che.
\newblock The tangent space in sub-{R}iemannian geometry.
\newblock In {\em Sub-{R}iemannian geometry}, volume 144 of {\em Progr. Math.}, pages 1--78. Birkh\"{a}user, Basel, 1996.

\bibitem[BMRT24a]{borza2024curvature}
Samuël Borza, Mattia Magnabosco, Tommaso Rossi, and Kenshiro Tashiro.
\newblock The curvature exponent of sub-finsler heisenberg groups.
\newblock {\em ArXiv 2407.14619}, 2024.

\bibitem[BMRT24b]{borza2024measure}
Samuël Borza, Mattia Magnabosco, Tommaso Rossi, and Kenshiro Tashiro.
\newblock Measure contraction property and curvature-dimension condition on sub-{F}insler {H}eisenberg groups.
\newblock {\em ArXiv 2402.14779}, 2024.

\bibitem[BT23]{borzatashiro}
Samuël Borza and Kenshiro Tashiro.
\newblock Measure contraction property, curvature exponent and geodesic dimension of sub-{F}insler {H}eisenberg groups.
\newblock {\em ArXiv 2305.16722}, 2023.

\bibitem[GJ15]{MR3388884}
R.~Ghezzi and F.~Jean.
\newblock Hausdorff volume in non equiregular sub-{R}iemannian manifolds.
\newblock {\em Nonlinear Anal.}, 126:345--377, 2015.

\bibitem[GJ16]{SRmeasures}
R.~Ghezzi and F.~Jean.
\newblock On measures in sub-{R}iemannian geometry.
\newblock {\em S\'eminaire TSG}, 33:17--46, 2015-2016.

\bibitem[GMS15]{GigMonSav}
Nicola Gigli, Andrea Mondino, and Giuseppe Savar\'{e}.
\newblock Convergence of pointed non-compact metric measure spaces and stability of {R}icci curvature bounds and heat flows.
\newblock {\em Proc. Lond. Math. Soc. (3)}, 111(5):1071--1129, 2015.

\bibitem[HKST15]{MR3363168}
Juha Heinonen, Pekka Koskela, Nageswari Shanmugalingam, and Jeremy~T. Tyson.
\newblock {\em Sobolev spaces on metric measure spaces}, volume~27 of {\em New Mathematical Monographs}.
\newblock Cambridge University Press, Cambridge, 2015.
\newblock An approach based on upper gradients.

\bibitem[Jea14]{MR3308372}
Fr\'{e}d\'{e}ric Jean.
\newblock {\em Control of nonholonomic systems: from sub-{R}iemannian geometry to motion planning}.
\newblock SpringerBriefs in Mathematics. Springer, Cham, 2014.

\bibitem[Jui21]{Juillet2020}
Nicolas Juillet.
\newblock Sub-{R}iemannian structures do not satisfy {R}iemannian {B}runn-{M}inkowski inequalities.
\newblock {\em Rev. Mat. Iberoam.}, 37(1):177--188, 2021.

\bibitem[LV09]{lott--villani2009}
John Lott and C\'{e}dric Villani.
\newblock Ricci curvature for metric-measure spaces via optimal transport.
\newblock {\em Ann. of Math. (2)}, 169(3):903--991, 2009.

\bibitem[Mit85]{MR806700}
John Mitchell.
\newblock On {C}arnot-{C}arath\'{e}odory metrics.
\newblock {\em J. Differential Geom.}, 21(1):35--45, 1985.

\bibitem[MR23a]{MR4562156}
Mattia Magnabosco and Tommaso Rossi.
\newblock Almost-{R}iemannian manifolds do not satisfy the curvature-dimension condition.
\newblock {\em Calc. Var. Partial Differential Equations}, 62(4):Paper No. 123, 2023.

\bibitem[MR23b]{magnabosco2023failure}
Mattia Magnabosco and Tommaso Rossi.
\newblock Failure of the curvature-dimension condition in sub-finsler manifolds.
\newblock {\em ArXiv 2307.01820}, 2023.

\bibitem[Oht07]{ohta2007}
Shin-ichi Ohta.
\newblock On the measure contraction property of metric measure spaces.
\newblock {\em Comment. Math. Helv.}, 82(4):805--828, 2007.

\bibitem[RS23]{MR4623954}
Luca Rizzi and Giorgio Stefani.
\newblock Failure of curvature-dimension conditions on sub-{R}iemannian manifolds via tangent isometries.
\newblock {\em J. Funct. Anal.}, 285(9):Paper No. 110099, 31, 2023.

\bibitem[Stu06a]{sturm2006-1}
Karl-Theodor Sturm.
\newblock On the geometry of metric measure spaces. {I}.
\newblock {\em Acta Math.}, 196(1):65--131, 2006.

\bibitem[Stu06b]{sturm2006}
Karl-Theodor Sturm.
\newblock On the geometry of metric measure spaces. {II}.
\newblock {\em Acta Math.}, 196(1):133--177, 2006.

\end{thebibliography}

\end{document}